\documentclass[a4paper,12pt]{amsart}

\usepackage{amsfonts}
\usepackage{amsmath}
\usepackage{amssymb}
\usepackage{graphicx}
\usepackage{mathrsfs}
\usepackage{hyperref}

\theoremstyle{plain}
 \newtheorem{theorem}{Theorem}[section]

\theoremstyle{definition}
 
 \newtheorem{definition}{Definition}[section]
\theoremstyle{remark}
 \newtheorem{remark}{Remark}[section]
 \numberwithin{equation}{section}

\title[Some Hermite-Hadamard type inequalities]{Some Hermite-Hadamard type inequalities in the class of hyperbolic p-convex functions}



\author[Dragomir]{\bfseries Silvestru Sever Dragomir}
\author[Torebek]{\bfseries Berikbol T. Torebek$^{\text{*}}$}

\address{Silvestru Sever Dragomir \newline College of Engineering and Science, Victoria University, \newline PO Box 14428, Melbourne City, MC 8001, Australia.}
\email{sever.dragomir@vu.edu.au}

\address{Berikbol T. Torebek \newline
Institute of Mathematics and Mathematical Modeling\newline 125 Pushkin str.,
050010 Almaty, Kazakhstan
\newline Al-Farabi Kazakh National University. \newline 71 Al-Farabi ave., 050040 Almaty, Kazakhstan}
\email{torebek@math.kz}

\thanks{$\text{*}$ Corresponding author. E-mail: torebek@math.kz}

\begin{document}

\vspace{18mm} \setcounter{page}{1} \thispagestyle{empty}

\begin{abstract}
In this paper, obtained some new class of Hermite-Hadamard and Hermite-Hadamard-Fej\'{e}r type inequalities via fractional integrals for the p-hyperbolic convex functions. It is shown that such inequalities are simple consequences of Hermite-Hadamard-Fej\'{e}r inequality for the p-hyperbolic convex function.\end{abstract}
\maketitle

\tableofcontents
\section{Introduction}
The inequalities for convex functions due to Hermite and Hadamard are found to be of great importance, for example, see \cite{DP00, PPT92}. According to the inequalities \cite{H1893, H1883},  \begin{itemize}
 \item if $u:\,I\rightarrow \mathbb{R}$ is a convex function on the interval $I\subset \mathbb{R}$ and $a,b\in I$ with $b>a,$ then
\begin{equation}\label{1.1}u\left(\frac{a+b}{2}\right)\leq \frac{1}{b-a}\int\limits^b_a u(y)dy \leq \frac{u(a)+u(b)}{2}.\end{equation}
 \end{itemize}
For a concave function $u$, the inequalities in \eqref{1.1} hold in the reversed direction.
\begin{definition} A function $u : [a, b]\subset \mathbb{R} \rightarrow \mathbb{R}$ is said to be convex if $$u(\mu x+(1-\mu)y)\leq \mu u(x)+(1-\mu)u(y)$$ for all $x, y\in [a,b]$ and $\mu\in [0,1].$ We call $u$ a  concave function if $(-u)$ is convex. \end{definition}
We note that Hadamard's inequality refines the concept of convexity,  and it  follows from Jensen's inequality. The classical Hermite-Hadamard inequality yields estimates for the mean value of a continuous convex function $u : [a, b] \rightarrow \mathbb{R}.$ The well-known inequalities dealing with the integral mean of a convex function $u$ are the Hermite-Hadamard inequalities or its weighted versions. They are also known as Hermite-Hadamard-Fej\'{e}r inequalities.

In \cite{F06}, Fej\'{e}r obtained the  weighted generalization of Hermite-Hadamard inequality \eqref{1.1} as follows.
\begin{itemize}
 \item Let $u:\, [a,b]\rightarrow \mathbb{R}$ be a convex function. Then the inequality
\begin{equation}\label{1.2}u\left(\frac{a+b}{2}\right)\int\limits^b_a v(y)dy\leq\int\limits^b_a u(y) v(y)dy \leq \frac{u(a)+u(b)}{2}\int\limits^b_a v(y)dy \end{equation} holds for a nonnegative, integrable function $v:\, [a,b]\rightarrow \mathbb{R}$, which is symmetric to $\frac{a+b}{2}.$
 \end{itemize}
Clearly, for $v(x)\equiv 1$ on $[a, b]$ we get \eqref{1.1}.

\subsection{Definitions of fractional integrals}\label{ss_fc}
Let us give basic definitions of fractional integrations of the different types.

\begin{definition}\cite{Kilbas} The left and right Riemann--Liouville
fractional integrals $I_{a+} ^\alpha$ and $I_{b-} ^\alpha$ of order $\alpha\in\mathbb R$ ($\alpha>0$) are given by
$$I_{a+} ^\alpha  \left[ f \right]\left( t \right) = {\rm{
}}\frac{1}{{\Gamma \left( \alpha \right)}}\int\limits_a^t {\left(
{t - s} \right)^{\alpha  - 1} f\left( s \right)} ds, \,\,\, t\in(a,b],$$
and
$$ I_{b-}^\alpha  \left[ f \right]\left( t \right) = {\rm{
}}\frac{1}{{\Gamma \left( \alpha \right)}}\int\limits_t^b {\left(
{s - t} \right)^{\alpha  - 1} f\left( s \right)} ds, \,\,\, t\in[a,b),$$
respectively. Here $\Gamma$ denotes the Euler gamma function.
\end{definition}
\begin{definition}\cite{KT16} Let $f\in L_1(a,b).$ The fractional integrals $\mathcal{I}^\alpha_{a+}$ and $\mathcal{I}^\alpha_{b-}$ of order $\alpha\in (0,1)$ are defined by
\begin{equation} \label{I-1}
\mathcal{I}^\alpha_{a+} u(x)=\frac{1}{\alpha}\int\limits^x_a \exp\left(-\frac{1-\alpha}{\alpha}(x-s)\right) u(s)ds,\, x>a
\end{equation}
and
\begin{equation}\label{I-2}
\mathcal{I}^\alpha_{b-} u(x)=\frac{1}{\alpha}\int\limits^b_x \exp\left(-\frac{1-\alpha}{\alpha}(s-x)\right) u(s)ds,\, x<b
\end{equation}
 respectively.\end{definition}

\subsection{Some generalizations of Hermite-Hadamard and Hermite-Hadamard-Fej\'{e}r inequalities}
Here we present some results on the generalization of the above inequalities.

In \cite{SSYB13}, Sarikaya et. al. represented Hermite-Hadamard inequality in Riemann-Liouville fractional integral forms as follows.
\begin{itemize}
 \item Let $u:\,[a,b]\rightarrow \mathbb{R}$ be a positive function and $u\in L^1([a,b]).$ If $u$ is a convex function on $[a,b],$ then the following inequalities for fractional integrals hold
\begin{equation}\label{FHH}u\left(\frac{a+b}{2}\right)\leq \frac{\Gamma(\alpha+1)}{2(b-a)^\alpha}\left[I^\alpha_au(b)+I^\alpha_bu(a)\right] \leq \frac{u(a)+u(b)}{2}\end{equation} with $\alpha>0.$
\end{itemize}
In \cite{I16}, I\c{s}can gave the following Hermite-Hadamard-Fej\'{e}r integral inequalities via fractional integrals:
\begin{itemize}
 \item  Let $u:\,[a,b]\rightarrow \mathbb{R}$ be convex function with $a < b$ and $u \in L^1([a, b]).$ If $v:\,[a,b]\rightarrow \mathbb{R}$ is nonnegative, integrable and symmetric to $(a+b)/2,$ then the following inequalities for fractional integrals hold \begin{equation}\label{FHHF}\begin{split}u\left(\frac{a+b}{2}\right)\left[I^\alpha_av(b)+I^\alpha_bv(a)\right]&\leq \left[I^\alpha_a(uv)(b)+I^\alpha_b(uv)(a)\right]\\& \leq \frac{u(a)+u(b)}{2}\left[I^\alpha_av(b)+I^\alpha_bv(a)\right] \end{split}\end{equation} with $\alpha>0.$
\end{itemize}
In \cite{KT16} the authors obtained the following generalizations of inequality \eqref{1.1} and \eqref{1.2}
\begin{itemize}
  \item Let $u: [a, b]\rightarrow \mathbb{R}$ and $u \in L^1 (a, b).$ If $u$ is a convex function on $[a, b],$ then the following inequalities for fractional integrals $\mathcal{I}^\alpha_{a+}$ and $\mathcal{I}^\alpha_{a+}$ hold:
\begin{equation}\label{FHH2} u\left(\frac{a+b}{2}\right)\leq \frac{1-\alpha}{2\left(1-\exp\left(-\rho\right)\right)}\left[\mathcal{I}^\alpha_a u(b)+\mathcal{I}^\alpha_b u(a)\right]\leq \frac{u(a)+u(b)}{2};
\end{equation}
  \item Let $u:\,[a,b]\rightarrow \mathbb{R}$ be convex and integrable function with $a<b.$ If $w:\,[a,b]\rightarrow \mathbb{R}$ is nonnegative, integrable and symmetric with respect to $\frac{a+b}{2},$ that is,  $w(a+b-x)=w(x),$ then the following inequalities hold
\begin{equation}\label{FHHF2}\begin{split} u\left(\frac{a+b}{2}\right)\left[\mathcal{I}^{\alpha}_a w(b)+\mathcal{I}^{\alpha}_b w(a)\right]& \leq \left[\mathcal{I}^{\alpha}_a \left(u w\right)(b)+\mathcal{I}^{\alpha}_b\left(u w\right)(a)\right]\\&\leq \frac{u(a)+u(b)}{2}\left[\mathcal{I}^{\alpha}_a w(b)+\mathcal{I}^{\alpha}_b w(a)\right].
\end{split}\end{equation}
\end{itemize}
Many generalizations and extensions of the Hermite-Hadamard and Hermite-Hadamard-Fej\'{e}r type inequalities were obtained for various classes of functions using fractional integrals; see \cite{C16,  CK17, HYT14, I16,  JS16, SSYB13, WLFZ12, ZW13} and references therein. In \cite{KS18} Kirane and Samet show that most of those results are particular cases of (or equivalent to) existing inequalities from the literature. These studies motivated us to consider a new class of functional inequalities for hyperbolic p-convex functions generalizing the classical Hermite-Hadamard and Hermite-Hadamard-Fej\'{e}r inequalities.

\subsection{Hyperbolic p-convex functions} We consider the hyperbolic functions of a real argument $x \in \mathbb{R}$ defined by
$$\sinh x:=\frac{e^x-e^{-x}}{2},$$ $$\cosh x:=\frac{e^x+e^{-x}}{2},$$ $$\tanh x:=\frac{\sinh x}{\cosh x},$$ $$\coth x:=\frac{\cosh x}{\sinh x}.$$
\begin{definition}\label{defHCF}\cite{D18a, D18b} We say that a function $f: I \rightarrow \mathbb{R}$ is hyperbolic p-convex (or sub H-function, according with \cite{A16}) on $I,$ if for any closed subinterval $[a, b]$ of $I$ we have
\begin{equation}\label{HCF1}f(x)\leq \frac{\sinh [p(b-x)]}{\sinh [p(b-a)]}f(a)+ \frac{\sinh [p(x-a)]}{\sinh [p(b-a)]}f(b)\end{equation} for all $x\in [a,b].$
\end{definition}
If the inequality \eqref{HCF1} holds with "$\geq$", then the function will be called hyperbolic p-concave on $I.$

Geometrically speaking, this means that the graph of $f$ on $[a, b]$ lies nowhere above the p-hyperbolic function determined by the equation
$$H(x)=H(x; a, b, f):=A\cosh (px)+B\sinh (px),$$ where $A$ and $B$ are chosen such that $H(a)=f(a)$ and $H(b)=f(b).$

If we take $x = (1-t) a + tb \in [a, b], t \in [0, 1],$ then the condition \eqref{HCF1} becomes
\begin{equation}\label{HCF1*}f((1-t)a+tb)\leq \frac{\sinh [p(1-t)(b-a)]}{\sinh [p(b-a)]}f(a)+ \frac{\sinh [pt(b-a)]}{\sinh [p(b-a)]}f(b)\end{equation} for any $t\in [0,1].$
We have the following properties of hyperbolic p-convex function on $I.$
\begin{description}
  \item[(i)] A hyperbolic p-convex function $f : I \rightarrow \mathbb{R}$ has finite right and left derivatives $f'_+(x)$ and $f'_-(x)$ at every point $x \in I$ and $f'_-(x)\leq f'_+(x).$ The
function $f$ is differentiable on $I$ with the exception of an at most countable set.
  \item[(ii)] A necessary and sufficient condition for the function $f : I \rightarrow \mathbb{R}$ to be hyperbolic p-convex function on $I$ is that it satisfies the gradient inequality
  $$f(y)\geq f(x)\cosh[p(y-x)]+K_{x,f}\sin[p(y-x)]$$ for any $x, y \in I,$ where $K_{x,f} \in \left[f'_-(x),f'_+(x)\right].$ If $f$ is differentiable at the point $x$ then $K_{x,f}=f'(x).$
  \item[(iii)] A necessary and sufficient condition for the function $f$ to be a hyperbolic p-convex in $I,$ is that the function $$\varphi(x)=f'(x)-p^2\int\limits_a^xf(t)dt$$ is nondecreasing on $I,$ where $a\in I.$
  \item[(iv)] Let $f : I \rightarrow \mathbb{R}$ be a two times continuously differentiable function on $I.$ Then $f$ is hyperbolic p-convex on $I$ if and only if for all $x\in I$ we have
$$f''(x)-p^2f(x) \geq 0.$$ For other properties of hyperbolic p-convex functions, see \cite{A16}.
\end{description}
Consider the function $f_r : (0,\infty) \rightarrow (0,\infty),$ $f_r (x) = x^r$ with $p \in \mathbb{R}\backslash \{0\}.$ If $r\in (-\infty, 0)\cup[1,\infty)$ the function is convex and if $r \in (0, 1)$ it is concave. We have for $r\in (-\infty, 0)\cup[1,\infty)$
$$f''_r(x)-p^2f_r(x)=r(r-1)x^{r-2}-p^2x^r=p^2x^{r-2}\left(\frac{r(r-1)}{p^2}-x^2\right),\,x>0.$$
We observe that $$f''_r(x)-p^2f_r(x) > 0 \,\,\,\textrm{for} \,\,\,x\in\left(0, \frac{\sqrt{r(r-1)}}{|p|}\right)$$ and $$f''_r(x)-p^2f_r(x) < 0\,\,\, \textrm{for} \,\,\, x\in\left(\frac{\sqrt{r(r-1)}}{|p|},\infty\right),$$ which shows that the power function $f_r$ for $r \in (-\infty,0)\cup[1,\infty)$ is hyperbolic p-convex on $\left(0, \frac{\sqrt{r(r-1)}}{|p|}\right)$ and hyperbolic p-concave on $\left(\frac{\sqrt{r(r-1)}}{|p|},\infty\right).$

If $r \in (0, 1),$ then $$f''_r(x)-p^2f_r(x) < 0$$ for any $x > 0,$ which shows that $f_r$ is hyperbolic p-concave on $(0,\infty).$

Consider the exponential function $$f_\mu (x) = \exp (\mu x)=e^{\mu x}$$ for $\mu\neq 0$ and $x \in \mathbb{R}.$ Then $$f''_\mu(x)-p^2f_\mu(x)=\mu^2e^{\mu x} -p^2e^{\mu x}=(\mu^2-p^2)e^{\mu x},\,x>0.$$
If $|\mu| > |p|,$ then $f_\mu$ is hyperbolic p-convex on $\mathbb{R}$ and if $|\mu| < |p|,$ then $f_\mu$ is hyperbolic p-concave on $\mathbb{R}.$

\subsection{Hermite-Hadamard and Hermite-Hadamard-Fej\'{e}r inequalities for hyperbolic p-convex functions}
In \cite{D18a}, Dragomir obtained Hermite-Hadamard type inequality in the class of hyperbolic p-convex functions as follows
\begin{theorem}\label{th1}
Assume that the function $u:\,I\rightarrow \mathbb{R}$ is hyperbolic p-convex function on $I.$ Then for any $a,b\in I$ we have
\begin{equation}\label{D1}\frac{2}{p}u\left(\frac{a+b}{2}\right)\sinh\left(\frac{p(b-a)}{2}\right)\leq \int\limits_a^bf(x)dx \leq \frac{u(a)+u(b)}{p}\tanh\left(\frac{p(b-a)}{2}\right).\end{equation}
\end{theorem}
\begin{remark}Note that, if $p\rightarrow 0$ in \eqref{D1} we get the classical Hermite-Hadamard inequality \eqref{1.1}.\end{remark}

Hermite-Hadamard-Fej\'{e}r type inequalities in the class of hyperbolic p-convex functions was proven in \cite{D18b}:
\begin{theorem}\label{th2} Assume that the function $u:\,I\rightarrow \mathbb{R}$ is hyperbolic p-convex on $I$ and
$a,b\in I.$ Assume also that $w:\,I\rightarrow \mathbb{R}$ is a positive, symmetric and integrable function on $[a, b],$ then we have
\begin{equation}\label{D2}\begin{split}u\left(\frac{a+b}{2}\right)&\int\limits_a^b\cosh\left[p\left(x-\frac{a+b}{2}\right)\right]w(x)dx\\& \leq \int\limits_a^b u(x)w(x)dx \\& \leq \frac{u(a)+u(b)}{2}\cosh^{-1}\frac{p(b-a)}{2}\int\limits_a^b\cosh\left[p\left(x-\frac{a+b}{2}\right)\right]w(x)dx.\end{split}\end{equation}\end{theorem}
\begin{remark}Note that, if $p\rightarrow 0$ in \eqref{D2} we get the classical Hermite-Hadamard-Fej\'{e}r inequality \eqref{1.2}.\end{remark}
\begin{theorem}\label{th3} Assume that the function $u:\,I\rightarrow \mathbb{R}$ is hyperbolic p-convex on $I$ and $a,b\in I.$ Assume also that $w:\,I\rightarrow \mathbb{R}$ is a positive, symmetric and integrable function on $[a, b],$ then
\begin{equation}\label{D3}\begin{split}&\int\limits_a^b u(x)w(x)dx\\& \leq \frac{u(a)+u(b)}{2}\cosh^{-1}\left(\frac{p(b-a)}{2}\right)\int\limits_a^b\cosh\left[p\left(x-\frac{a+b}{2}\right)\right]w(x)dx \\&+\frac{u(a)-u(b)}{2}\sinh^{-1}\left(\frac{p(b-a)}{2}\right)\int\limits_a^b\sinh\left[p\left(x-\frac{a+b}{2}\right)\right]w(x)dx.\end{split}\end{equation}
\end{theorem}

\section{Main results}
In this section, we formulate the main results of the paper.
\subsection{Fractional analogues of Hermite-Hadamard inequality for hyperbolic p-convex functions}
Our first observation is formulated by the following theorem.
\begin{theorem}\label{th4}
Assume that the function $u:\,I\rightarrow \mathbb{R}$ is hyperbolic p-convex function on $I.$ Then for any $a,b\in I$ we have
\begin{equation}\label{D4}u\left(\frac{a+b}{2}\right)\mathcal{C}_{\alpha}(1) \leq I^{\alpha}_{a+}u(b)+ I^{\alpha}_{b-}u(a) \leq \frac{u(a)+u(b)}{2}\cosh^{-1}\left(p(b-a)\right)\mathcal{C}_\alpha(1),\end{equation}
where $\mathcal{C}_{\alpha}(1):=\int\limits_a^b\cosh\left[p\left(x-\frac{a+b}{2}\right)\right]\frac{(b-x)^{\alpha-1}+(x-a)^{\alpha-1}}{\Gamma(\alpha)}dx.$
\end{theorem}
\begin{proof}Let us suppose that all assumptions of Theorem are satisfied. Let us define the function $w$ of Theorem \ref{th2} by
$$w(x)=\frac{1}{\Gamma(\alpha)}\left[(b-x)^{\alpha-1}+(x-a)^{\alpha-1}\right],\, a<x<b.$$ Clearly, $w$ is a positive, symmetric and integrable function
on $[a,b].$ Moreover, for all $x\in(a,b),$ we get $$w(a+b-x)=(b-(a+b-x))^{\alpha-1}+((a+b-x)-a)^{\alpha-1}=w(x).$$ Moreover, we have
\begin{align*}\int\limits_a^b u(x)w(x)dx&=\frac{1}{\Gamma(\alpha)}\int\limits_a^b(b-x)^{\alpha-1} u(x)dx+\frac{1}{\Gamma(\alpha)}\int\limits_a^b(x-a)^{\alpha-1} u(x)dx\\&=I^{\alpha}_{b-}u(a)+I^{\alpha}_{a+}u(b).\end{align*} Therefore, from \eqref{D2} we obtain inequality \eqref{D4}.
\end{proof}
\begin{remark} Inequalities \eqref{1.1} and \eqref{FHH} are special cases of inequality \eqref{D4}.
\begin{itemize}
  \item If $p\rightarrow 0$ we have $\lim\limits_{p\rightarrow 0}\mathcal{C}^1_\alpha(1)=\frac{2(b-a)^\alpha}{\Gamma(\alpha+1)}.$ In this case, inequality \eqref{D4} coincide with the inequality \eqref{FHH};
  \item If $p\rightarrow 0$ and $\alpha\rightarrow 1$ in \eqref{D4}, then we have classical Hermite-Hadamard inequality \eqref{1.1}.
\end{itemize}
\end{remark}
\begin{theorem}\label{th5}
Assume that the function $u:\,I\rightarrow \mathbb{R}$ is hyperbolic p-convex function on $I.$ Then for any $a,b\in I$ we have
\begin{equation}\label{D5}u\left(\frac{a+b}{2}\right)\mathcal{C}^2_{\alpha}(1) \leq \mathcal{I}^{\alpha}_{a+}u(b)+ \mathcal{I}^{\alpha}_{b-}u(a) \leq \frac{u(a)+u(b)}{2}\cosh^{-1}\left(p(b-a)\right)\mathcal{C}^2_\alpha(1),\end{equation}
where $\mathcal{C}^2_{\alpha}(1):=\int\limits_a^b\cosh\left[p\left(x-\frac{a+b}{2}\right)\right]\frac{\exp\left(-\frac{1-\alpha}{\alpha}(b-x)\right) +\exp\left(-\frac{1-\alpha}{\alpha}(x-a)\right)}{\alpha}dx.$
\end{theorem}
\begin{proof}Suppose that all assumptions of Theorem are satisfied. Let us define the function $w$ of Theorem \ref{th2} by
$$w(x)=\frac{1}{\alpha}\left[\exp\left(-\frac{1-\alpha}{\alpha}(b-x)\right) +\exp\left(-\frac{1-\alpha}{\alpha}(x-a)\right)\right],\, a<x<b.$$ Clearly, $w$ is a positive, symmetric and integrable function
on $[a,b].$ Moreover, for all $x\in(a,b),$ we get \begin{align*}\alpha w(a+b-x)&=\left[\exp\left(-\frac{1-\alpha}{\alpha}(b-(a+b-x))\right) +\exp\left(-\frac{1-\alpha}{\alpha}((a+b-x)-a)\right)\right]\\& =\left[\exp\left(-\frac{1-\alpha}{\alpha}(x-a)\right) +\exp\left(-\frac{1-\alpha}{\alpha}(b-x)\right)\right]=\alpha w(x).\end{align*} Moreover, we have
\begin{align*}&\int\limits_a^b u(x)w(x)dx\\&=\frac{1}{\alpha}\int\limits_a^b\exp\left(-\frac{1-\alpha}{\alpha}(b-x)\right) u(x)dx+\frac{1}{\alpha}\int\limits_a^b \exp\left(-\frac{1-\alpha}{\alpha}(x-a)\right) u(x)dx\\&=\mathcal{I}^{\alpha}_{b-}u(a)+\mathcal{I}^{\alpha}_{a+}u(b).\end{align*} Therefore, from \eqref{D2} we obtain inequality \eqref{D5}.
\end{proof}
\begin{remark} Inequalities \eqref{1.1} and \eqref{FHH} are special cases of inequality \eqref{D5}.
\begin{itemize}
  \item If $p\rightarrow 0$ we have $\lim\limits_{p\rightarrow 0}\mathcal{C}^2_\alpha(1)=\frac{2\exp\left(-\frac{1-\alpha}{\alpha}(b-a)\right)}{1-\alpha}.$ In this case, inequality \eqref{D5} coincide with the inequality \eqref{FHH2};
  \item If $p\rightarrow 0$ and $\alpha\rightarrow 1$ in \eqref{D5}, then we have classical Hermite-Hadamard inequality \eqref{1.1}.
\end{itemize}
\end{remark}

\subsection{Fractional analogues of Hermite-Hadamard-Fej\'{e}r inequality for hyperbolic p-convex functions}
\begin{theorem}\label{th6} Assume that the function $u:\,I\rightarrow \mathbb{R}$ is hyperbolic p-convex on $I$ and
$a,b\in I.$ Assume also that $v:\,I\rightarrow \mathbb{R}$ is a positive, symmetric and integrable function on $[a, b],$ then we have
\begin{equation}\label{D6}\begin{split}u\left(\frac{a+b}{2}\right)\mathcal{C}^1_\alpha(v) &\leq I^\alpha_{a+}\left[u(b)v(b)\right]+I^\alpha_{b-}\left[u(a)v(a)\right] \\& \leq \frac{u(a)+u(b)}{2}\cosh^{-1}\frac{p(b-a)}{2}\mathcal{C}^1_\alpha(v),\end{split}\end{equation}
where $\mathcal{C}^1_{\alpha}(v):=\int\limits_a^b\cosh\left[p\left(x-\frac{a+b}{2}\right)\right]\frac{(b-x)^{\alpha-1}+(x-a)^{\alpha-1}}{\Gamma(\alpha)}v(x)dx.$
\end{theorem}
\begin{proof}Let us suppose that all assumptions of Theorem are satisfied. Let us define the function $w$ of Theorem \ref{th2} by
$$w(x)=\frac{v(x)}{\Gamma(\alpha)}\left[(b-x)^{\alpha-1}+(x-a)^{\alpha-1}\right],\, a<x<b.$$ Clearly, $w$ is a positive, symmetric and integrable function
on $[a,b].$ Moreover, for all $x\in(a,b),$ we get $w(a+b-x)=w(x).$ Moreover, we have
\begin{align*}\frac{1}{\Gamma(\alpha)}\int\limits_a^b u(x)v(x)\left[(b-x)^{\alpha-1}+(x-a)^{\alpha-1}\right]dx&=I^\alpha_{a+}\left[u(b)v(b)\right]+I^\alpha_{b-}\left[u(a)v(a)\right].\end{align*} Therefore, from \eqref{D2} we obtain inequality \eqref{D6}.
\end{proof}
\begin{remark} Inequalities \eqref{1.2} and \eqref{FHHF} are special cases of inequality \eqref{D6}.
\begin{itemize}
  \item If $p\rightarrow 0$ we have $\lim\limits_{p\rightarrow 0}\mathcal{C}^1_\alpha(v)=I^\alpha_{a+}v(b)+I^\alpha_{b-}v(a).$ In this case, inequality \eqref{D6} coincide with the inequality \eqref{FHHF};
  \item If $p\rightarrow 0$ and $\alpha\rightarrow 1$ in \eqref{D6}, then we have classical Hermite-Hadamard-Fej\'{e}r inequality \eqref{1.2}.
\end{itemize}
\end{remark}
The following theorems are proved similarly.
\begin{theorem}\label{th7} Assume that the function $u:\,I\rightarrow \mathbb{R}$ is hyperbolic p-convex on $I$ and
$a,b\in I.$ Assume also that $v:\,I\rightarrow \mathbb{R}$ is a positive, symmetric and integrable function on $[a, b],$ then we have
\begin{equation}\label{D7}\begin{split}u\left(\frac{a+b}{2}\right)\mathcal{C}^2_\alpha(v) &\leq \mathcal{I}^\alpha_{a+}\left[u(b)v(b)\right]+\mathcal{I}^\alpha_{b-}\left[u(a)v(a)\right] \\& \leq \frac{u(a)+u(b)}{2}\cosh^{-1}\frac{p(b-a)}{2}\mathcal{C}^2_\alpha(v),\end{split}\end{equation}
where $\mathcal{C}^2_{\alpha}(v):=\int\limits_a^b\cosh\left[p\left(x-\frac{a+b}{2}\right)\right]\frac{\exp\left(-\frac{1-\alpha}{\alpha}(b-x)\right) +\exp\left(-\frac{1-\alpha}{\alpha}(x-a)\right)}{\alpha}v(x)dx.$
\end{theorem}
\begin{remark} Inequalities \eqref{1.2} and \eqref{FHHF} are special cases of inequality \eqref{D7}.
\begin{itemize}
  \item If $p\rightarrow 0$ we have $\lim\limits_{p\rightarrow 0}\mathcal{C}^2_\alpha(v)=\mathcal{I}^\alpha_{a+}v(b)+\mathcal{I}^\alpha_{b-}v(a).$ In this case, inequality \eqref{D7} coincide with the inequality \eqref{FHHF2};
  \item If $p\rightarrow 0$ and $\alpha\rightarrow 1$ in \eqref{D7}, then we have classical Hermite-Hadamard-Fej\'{e}r inequality \eqref{1.2}.
\end{itemize}
\end{remark}
\begin{theorem}\label{th8} Assume that the function $u:\,I\rightarrow \mathbb{R}$ is hyperbolic p-convex on $I$ and $a,b\in I.$ Assume also that $v:\,I\rightarrow \mathbb{R}$ is a positive, symmetric and integrable function on $[a, b],$ then
\begin{equation}\label{D8}\begin{split}I^\alpha_{a+}\left[u(b)v(b)\right]+I^\alpha_{b-}\left[u(a)v(a)\right]& \leq \frac{u(a)+u(b)}{2}\cosh^{-1}\left(\frac{p(b-a)}{2}\right)\mathcal{C}_\alpha^1(v) \\&+\frac{u(a)-u(b)}{2}\sinh^{-1}\left(\frac{p(b-a)}{2}\right)\mathcal{S}^1_\alpha(v),\end{split}\end{equation}
where $$\mathcal{S}^1_\alpha(v):=\int\limits_a^b\sinh\left[p\left(x-\frac{a+b}{2}\right)\right]\frac{(b-x)^{\alpha-1}+(x-a)^{\alpha-1}}{\Gamma(\alpha)}v(x)dx.$$
\end{theorem}
\begin{remark} If $\alpha\rightarrow 1,$ we have $$\lim\limits_{\alpha\rightarrow 1}\mathcal{C}^1_\alpha(v)=\int\limits_a^b\cosh\left[p\left(x-\frac{a+b}{2}\right)\right]v(x)dx$$ and $$\lim\limits_{\alpha\rightarrow 1}\mathcal{S}^1_\alpha(v)=\int\limits_a^b\sinh\left[p\left(x-\frac{a+b}{2}\right)\right]v(x)dx.$$ In this case, inequality \eqref{D8} coincide with the inequality \eqref{D3}.
\end{remark}
\begin{theorem}\label{th9} Assume that the function $u:\,I\rightarrow \mathbb{R}$ is hyperbolic p-convex on $I$ and $a,b\in I.$ Assume also that $v:\,I\rightarrow \mathbb{R}$ is a positive, symmetric and integrable function on $[a, b],$ then
\begin{equation}\label{D9}\begin{split}\mathcal{I}^\alpha_{a+}\left[u(b)v(b)\right]+\mathcal{I}^\alpha_{b-}\left[u(a)v(a)\right]& \leq \frac{u(a)+u(b)}{2}\cosh^{-1}\left(\frac{p(b-a)}{2}\right)\mathcal{C}_\alpha^2(v) \\&+\frac{u(a)-u(b)}{2}\sinh^{-1}\left(\frac{p(b-a)}{2}\right)\mathcal{S}^2_\alpha(v),\end{split}\end{equation}
where $\mathcal{S}^2_\alpha(v):=\int\limits_a^b\sinh\left[p\left(x-\frac{a+b}{2}\right)\right]\frac{\exp\left(-\frac{1-\alpha}{\alpha}(b-x)\right) +\exp\left(-\frac{1-\alpha}{\alpha}(x-a)\right)}{\alpha}v(x)dx.$
\end{theorem}
\begin{remark} If $\alpha\rightarrow 1,$ we have $$\lim\limits_{\alpha\rightarrow 1}\mathcal{C}^2_\alpha(v)=\int\limits_a^b\cosh\left[p\left(x-\frac{a+b}{2}\right)\right]v(x)dx$$ and $$\lim\limits_{\alpha\rightarrow 1}\mathcal{S}^2_\alpha(v)=\int\limits_a^b\sinh\left[p\left(x-\frac{a+b}{2}\right)\right]v(x)dx.$$ In this case, inequality \eqref{D9} coincide with the inequality \eqref{D3}.
\end{remark}
\section*{Acknowledgements} The research of Torebek is financially supported by a grant No.AP05131756 from the Ministry of Science and Education of the Republic of Kazakhstan. No new data was collected or generated during the course of research

\end{document}